\documentclass[12pt]{article}
\usepackage{amsmath}
\usepackage{amsthm}
\usepackage{amsfonts}
\usepackage{eucal}

\usepackage[sc]{mathpazo} 
\usepackage[scaled]{helvet} 
\usepackage{eulervm} 
\usepackage[pdftex]{hyperref}

\theoremstyle{plain}

\newtheorem{thm}{Theorem}
\newtheorem{coroll}{Corollary}

\theoremstyle{definition}

\begin{document}
\title{A Note on the generating function of p-Bernoulli numbers}%
\author{Markus Kuba}
\date{\today}

\maketitle
\begin{abstract}
We use analytic combinatorics to give a direct proof of the closed formula for the generating function of $p$-Bernoulli numbers.
\end{abstract}


\section{Introduction}
Rahmani~\cite{R} introduced $p$-Bernoulli numbers $B_{n,p}$, $p\ge 0$, generalizing the ordinary
Bernoulli number $B_n$, defined by their exponential generating function
\[
\sum_{n\ge 0}B_n\frac{t^n}{n!}=\frac{t}{e^t-1}.
\]
The $p$-Bernoulli numbers $B_{n,p}$ where defined using an infinite matrix, 
satifying $B_{0,p} = 1$ for $p\ge 0$, $B_{n,0}=B_n$ and the recurrence relation
\begin{equation}
\label{rec}
B_{n+1,p} = p B_{n,p} - \frac{(p + 1)^2}{p + 2}B_{n,p+1},\quad p\ge 0,\ n\ge 0.
\end{equation}
A closed form expression was already obtained by Rahmani~\cite{R}. Kargin and Rahmani~\cite{KR} derived the generating functions
\[
f_p(t):=\sum_{n\ge 0}B_{n,p}\frac{t^n}{n!},
\]
$p\ge 0$, using iterated integrals and induction as their main result and used it to give interesting evaluations of the $p$-Bernoulli numbers. Prodinger and Selkirk~\cite{PS} recently gave an elementary proof of the following result above (using induction in a neat way). 
\begin{thm}
\label{themain}
The exponential generating functions $f_p(t)$ of $p$-Bernoulli numbers $B_{n,p}$, defined via the recurrence relation~\eqref{rec},
are given by
\[
f_p(t)=\frac{(p+1)e^{pt}(t-H_p)}{(e^t-1)^{p+1}}+(p+1)\sum_{k=1}^{p}\binom{p}{k}\frac{H_k}{(e^t-1)^{k+1}}.
\]
Here and throughout this work $H_k=\sum_{j=1}^{k}\frac1j$ denotes the $k$-th harmonic number.
\end{thm}

\medskip 

The purpose of this short note is to use analytic combinatorics and generating functions to give an alternative direct proof, which is free of induction.
\begin{thm}
\label{theMein}
The bivariate generating function $G(z,t)=\sum_{p\ge 0}\sum_{n\ge 0} B_{n,p}\frac{t^n}{n!}\frac{z^p}{p+1}$
is given by 
\[
G(z,t)=\frac{t+\log(1-z)}{e^t(1-z)-1}.
\]
\end{thm}
From the result above we easily obtain the earlier result as a Corollary using $f_p(t)=(p+1)[z^p]G(z,t)$.
\begin{coroll}
\label{coroll1}
The univariate generating function $f_p(t)=\sum_{n\ge 0}B_{n,p}\frac{t^n}{n!}$
is given by 
\[
f_p(t)=\frac{(p+1)t e^{pt}}{(e^t-1)^{p+1}}- (p+1)\sum_{k=1}^{p}\frac{1}{k}\frac{e^{{p-k}t}}{(e^t-1)^{p-k+1}}
\]
and concides with the expression in Theorem~\ref{themain}.
\end{coroll}
\begin{proof}[Proof of Corollary~\ref{coroll1}]
We need the expansions
\[
\frac{1}{e^t(1-z)-1}=\frac{1}{e^t-1-ze^t}=\frac{1}{(e^t-1)\big(1-z\frac{e^t}{e^t-1}\big)}
=\frac{1}{e^t-1}\sum_{k\ge 0}\frac{z^k e^{kt}}{(e^t-1)^k}
\]
as well as
\[
\frac{\log(1-z)}{e^t(1-z)-1}=-\sum_{p\ge 1}z^p\sum_{k=1}^{p}\frac{1}{k}\frac{e^{{p-k}t}}{(e^t-1)^{p-k+1}}.
\]
Thus, extraction of coefficients
\[
f_p(t)=(p+1)[z^p]G(z,t)=(p+1)[z^p]\frac{t}{e^t(1-z)-1} + (p+1)[z^p]\frac{\log(1-z)}{e^t(1-z)-1}
\]
gives the stated expression. In order to observe the equality of the two different expressions, we state the following
\begin{align*}
-\sum_{k=1}^{p}\frac1k \frac{x^{p-k}}{(x-1)^{p-k+1}}&=\sum_{\ell=0}^{p-1}\binom{p}\ell \frac{1}{(x-1)^{\ell+1}}(H_\ell-H_p)\\
&= \sum_{\ell=0}^{p}\binom{p}\ell \frac{1}{(x-1)^{\ell+1}}H_\ell- H_p\frac{x^p}{(x-1)^{p+1}}.
\end{align*}
The second expression is simply the Laurent series expansion of the first around $x=1$. 
The prove our statement we write $x^k=(x-1+1)^k$ and use the binomial theorem:
\[
\sum_{k=1}^{p}\frac1k \frac{x^{p-k}}{(x-1)^{p-k+1}}
= \sum_{k=1}^{p}\frac1k \frac{1}{(x-1)^{p-k+1}}\sum_{\ell=0}^{p-k}\binom{p-k}{\ell}(x-1)^{p-k-\ell}.
\]
Cancellations and changing the order of summation give
\[
\sum_{\ell=0}^{p-1}\frac1{(x-1)^{\ell+1}}\sum_{k=1}^{p-\ell}\binom{p-k}{\ell}\frac1k.
\]
Application of the well-known identity, 
\[
\sum_{k=j}^{n-1}\binom{k}{j}\frac1{n-k}=\binom{n}{j}(H_n-H_j),
\]
which follows directly from 
\begin{align*}
 \sum_{k=j}^{n-1}\binom{k}{j}\frac{1}{n-k}&=
    \sum_{k=j}^{n-1}[z^{k}]\frac{z^{j}}{(1-z)^{j+1}}
    [z^{n-k}]\log\Bigl(\frac{1}{1-z}\Bigr)\\
		&=[z^{n}]\frac{z^{j}}{(1-z)^{j+1}}%
    \log\Bigl(\frac{1}{1-z}\Bigr)\\
		&= [z^{n-j}]\frac{1}{(1-z)^{j+1}}\log\Bigl(\frac{1}{1-z}\Bigr)=
    \binom{n}{j}(H_n-H_j)
\end{align*}
gives the desired intermediate result. Note that the last identity follows
from 
\[
\frac{1}{(1-z)^{m+1}}=\sum_{n\ge 0}\binom{n+m}{n}z^n
\]
after differentiation with respect to $m$. 

\smallskip

Finally, we extend the range of summation to $\ell=p$ 
and use
\[
\sum_{\ell=0}^{p}\binom{p}\ell \frac{1}{(x-1)^{\ell+1}}
=\frac1{x-1} \Big(1+\frac1{x-1}\Big)^{p}=\frac{x^p}{(x-1)^{p+1}}.
\]
With $x=e^{t}$ we get the required identity. 
\end{proof}

\begin{proof}[Proof of Theorem~\ref{theMein}]
Following~\cite{PS} we translate the recurrence relation~\eqref{rec} into a recurrence relation
for $f_p(t)$ by multiplication with $\frac{t^n}{n!}$ and summing over $n\ge 0$.
\begin{equation}
f'_p(t)=p f_p(t) -\frac{(p+1)^2}{p+2}f_{p+1}(t),\quad p\ge 0.
 \label{diffrec}
\end{equation}
In order to obtain Theorem~\ref{themain} we use generating functions. We proceed by normalizing the recurrence relation to avoid a more difficult problem. Let 
\[
g_p(t):=\frac{f_p(t)}{p+1}.
\]
Thus, the recurrence relation~\eqref{diffrec} simplifies to
\[
g'_p(t)=p\cdot g_p(t)-(p+1) g_{p+1}(t),\quad p\ge 0.
\]
We introduce the bivariate generating function
\[
G(z,t):=\sum_{p\ge 0}g_p(t)z^p=\sum_{p\ge 0}f_p(t)\frac{z^p}{p+1}=\sum_{p\ge 0}\sum_{n\ge 0} B_{n,p}\frac{t^n}{n!}\frac{z^p}{p+1}.
\]
and obtain
\[
G_t(z,t)=z G_z(z,t) - G_z(z,t).
\]
Initial values are given by
\[
G(z,0)=\sum_{p\ge 0}g_p(0)z^p
=\sum_{p\ge 0}\frac1{p+1}z^p
=\frac1z \log(1/(1-z))
\] 
and
\[
G(0,t)=g_0(t)=f_0(t)=\frac{t}{e^t-1}.
\]
Thus, our desired generating function satisfies a first order linear partial differential equation:
\[
G_t(z,t) + (1-z)G_z(z,t) =0. 
\]
We use the method of characteristics to solve the differential equation. Let $t=t(x)$ and $z=z(x)$. A first integral of
\[
\frac{d}{dx}t=1,\quad \frac{d}{dx}z=1-z
\]
or equivalently 
\[
t'(z)=\frac{1}{1-z}
\]
is given by 
\[
\xi=\xi(z,t)=t+\log(1-z).
\]
Hence, any continuous function $C(\xi)=C(t+\log(1-z))$ is 
a solution of the equation.
For $z=0$ we get
\[
C(t)=\frac{t}{e^t-1}.
\]
Thus
\[
C(\log(1-z))=\frac{\log(1-z)}{\exp(\log(1-z))-1}
=\frac{\log(1-z)}{1-z-1}=\frac1z \log(1/(1-z)),
\]
matching the initial value for $t=0$. 
Therefore, our generating function is given by
\[
G(z,t)=\frac{t+\log(1-z)}{\exp\Big(t+\log(1-z)\Big)-1}
=\frac{t+\log(1-z)}{e^t(1-z)-1}.
\]
\end{proof}

\end{document}